\documentclass[10pt, reqno]{amsart}
\usepackage[dvipsnames]{xcolor}
\usepackage{amsmath}
\usepackage{amsfonts}
\usepackage{amssymb}
\usepackage{amsthm}
\usepackage{tikz}
\usepackage[sc]{mathpazo}
\usepackage{bm}
\usepackage{graphicx}

\usepackage{algorithm}
\usepackage{algpseudocode}

\usepackage{enumitem}
    \setenumerate{itemsep=5pt} 
    \setitemize{itemsep=5pt} 

\usepackage[font=small,labelfont=bf]{caption}

\theoremstyle{plain}
\newtheorem{theorem}{Theorem}
\numberwithin{theorem}{section}
\newtheorem{proposition}[theorem]{Proposition}
\newtheorem{lemma}[theorem]{Lemma}
\newtheorem{corollary}[theorem]{Corollary}

\theoremstyle{definition}

\newtheorem{remark}[theorem]{Remark}
\newtheorem{example}[theorem]{Example}

\numberwithin{equation}{section}

\newcommand{\superdiag}{\operatorname{superdiag}}

\newcommand{\adj}{\operatorname{adj}}

\newcommand{\Z}{\mathbb{Z}}

\newcommand{\K}{\mathbb{K}}
\newcommand{\M}{\mathrm{M}}

\newcommand{\N}{\mathbb{N}}
\newcommand{\Q}{\mathbb{Q}}

\renewcommand{\P}{\mathcal{P}}
\newcommand{\p}{\mathfrak{p}}

\newcommand{\mdim}{\operatorname{dim}_{\mathrm{mat}}}
\renewcommand{\pmod}[1]{\,\,(\operatorname{mod} #1)}
\newcommand{\0}{{\color{lightgray}0}}

\renewcommand{\S}{\mathcal{S}}
\newcommand{\SL}{\mathrm{SL}}

\newcommand{\tr}{\operatorname{tr}}

\renewcommand{\vec}[1]{{\bm{#1}}}

\newcommand{\minimatrix}[4]{\Big[ \begin{smallmatrix} #1 & #2 \\ #3 & #4 \end{smallmatrix} \Big]}

\newcommand{\semigroup}[1]{\langle #1 \rangle}

\allowdisplaybreaks
\usepackage{hyperref}

\begin{document}

\title[Numerical semigroups from rational matrices I]{Numerical semigroups from rational matrices I: power-integral matrices and nilpotent representations}
\author[A.~Chhabra]{Arsh Chhabra}
\address{Department of Mathematics and Statistics, Pomona College, 610 N. College Ave., Claremont, CA 91711, USA}
\email{acaa2021@mymail.pomona.edu}

\author[S.R.~Garcia]{Stephan Ramon Garcia}
\address{Department of Mathematics and Statistics, Pomona College, 610 N. College Ave., Claremont, CA 91711, USA}
\email{stephan.garcia@pomona.edu}
\urladdr{\url{https://stephangarcia.sites.pomona.edu/}}

\author[F.~Zhang]{Fangqian Zhang}
\address{Department of Mathematics
South Hall, Room 6607
University of California
Santa Barbara, CA 93106-3080, USA}
\email{fangqian@ucsb.edu}

\author[H.~Zhang]{Hechun Zhang}
\address{Department of Mathematical Sciences, Tsinghua University, Beijing, 100084, P. R. China.}
\email{hczhang@tsinghua.edu.cn}

\thanks{SRG was partially supported by NSF grant DMS-2054002. Hechun Zhang was partially supported by National Natural Science Foundation of China grants No. 12031007 and No. 11971255}

\begin{abstract}
Our aim in this paper is to initiate the study of exponent semigroups for rational matrices.  We prove that every numerical semigroup is the exponent semigroup of some rational matrix.  We also obtain lower bounds on the size of such matrices and discuss the related class of power-integral matrices.
\end{abstract}

%
%
%
%
%
%
%

\maketitle

\section{Introduction}
A \emph{numerical semigroup} is a subsemigroup  of $\N = \{0,1,2,\ldots\}$ with finite complement  \cite{Assi, Rosales}.
The largest natural number not in $S$ is the \emph{Frobenius number} $g(S)$ of $S$.
For each numerical semigroup $S$, there is a unique minimal system of \emph{generators} 
$1 \leq n_1 < n_2 < \cdots < n_k$ 
such that $\gcd(n_1,n_2,\ldots,n_k) = 1$ and
\begin{equation*}
S = \semigroup{n_1,n_2,\ldots,n_k}
= \{ a_1 n_1 + a_2 n_2 + \cdots + a_k n_k : a_1,a_2,\ldots,a_k \in \N\},
\end{equation*}
the commutative subsemigroup of $\N$ generated by $n_1,n_2,\ldots,n_k$ \cite[Thm.~2.7]{Rosales}.
Here $e(S) = k$ is the \emph{embedding dimension} of $S$ and $m(S) = n_1$ is the \emph{multiplicity} of $S$.

Let $\M_d(\cdot)$ be the set of $d \times d$ matrices with entries in the specified set, which in this paper is usually 
the set $\Z$ of integers or the set $\Q$ of rational numbers. 
The \emph{exponent semigroup} of $A \in \M_d(\Q)$ is 
\begin{equation*}
\S(A) = \{ n \in \N : A^n \in \M_d(\Z)\}.
\end{equation*}
Since $A^{j+k} = A^j A^k$ for $j,k \in \N$
and $A^0 = I \in \M_d(\Z)$, it is clear that $\S(A)$ is a subsemigroup of $\N$.
If $\S(A) = \{0\}$, then $\S(A)$ is \emph{trivial}; otherwise it is \emph{nontrivial}.

For example, the exponent semigroup of 
\begin{equation}\label{eq:NuggetMatrix}
\left[
\begin{smallmatrix}
 \frac{31837}{256} & \frac{31899}{256} & -\frac{9751}{128} & -\frac{3857}{32} & \frac{7703}{64} & -\frac{10313}{256} \\[3pt]
 -\frac{140489}{256} & -\frac{233823}{256} & \frac{33843}{128} & \frac{24487}{32} & -\frac{35891}{64} & \frac{119973}{256} \\[3pt]
 \frac{101403}{128} & \frac{176689}{128} & -\frac{21255}{64} & -\frac{34955}{32} & \frac{25351}{32} & -\frac{95767}{128} \\[3pt]
 -\frac{109715}{256} & -\frac{180477}{256} & \frac{25637}{128} & \frac{4647}{8} & -\frac{27685}{64} & \frac{93303}{256} \\[3pt]
 \frac{30963}{64} & \frac{61421}{64} & -\frac{5885}{32} & -\frac{2991}{4} & \frac{7933}{16} & -\frac{35063}{64} \\[3pt]
 -\frac{12355}{128} & -\frac{15989}{128} & \frac{4369}{64} & \frac{2125}{16} & -\frac{3345}{32} & \frac{5687}{128} \\
\end{smallmatrix}
\right]
\end{equation}
is $\semigroup{6,9,20}$, the McNugget monoid; no matrix of smaller dimension
can produce this semigroup (Example \ref{Example:Nugget}).

The most obvious question about exponent semigroups is the following:
\begin{equation*}
\textit{Is every semigroup in $\N$ the exponent semigroup for a rational matrix?}
\end{equation*}
We answer this question in the affirmative (Theorem \ref{Theorem:Nilpotent}).  In fact, our proof is constructive and provides
a fast algorithm to find such representations.  An easy corollary (Corollary \ref{Corollary:General}) asserts that
any subsemigroup in $\N$, numerical or not, is the exponent semigroup of a rational matrix.

These results ensure the \emph{matricial dimension}  
\begin{equation*}
\mdim S = \min\{ d \geq 1 : \text{there is an $A \in \M_d(\Q)$ such that $S = \S(A)$}\}
\end{equation*}
of a semigroup $S \subseteq \N$ is well defined.  We address the relationship between $\mdim S$
and several of the important quantities introduced above.
For example, Theorem \ref{Theorem:EmbeddingSharp} says that $\mdim S \geq m(S)$
if $S$ is a symmetric numerical semigroup.

\begin{remark}
There is a dynamical interpretation of exponent semigroups.
Define $f:\N \to [0,1)^{d^2}$ by $f(n) = A^n\pmod{1}$ and observe
\begin{equation*}
\S(A) = f^{-1}( \{0 \}) =   \{ n \in \N : A^n \equiv 0 \pmod{1} \},
\end{equation*}
so $\S(A) \neq \{0\}$ if and only if the forward orbit of $A$ in the $d^2$-dimensional torus returns to $0$.
The case $d=1$ is uninteresting for us since $\S(A) = \{0\}$ or $\N$.  
\end{remark}

This paper is organized as follows.
Section \ref{Section:Preliminaries} contains preliminary results and instructive examples.
In Section \ref{Section:PowerIntegral}, we discuss power-integral matrices, a class of matrices closely connected to exponent semigroups. 
We study cyclicity for exponent semigroups in Section \ref{Section:Cyclic}.
We obtain bounds on the matricial dimension
for symmetric and pseudosymmetric semigroups in Section \ref{Section:Symmetric}.
In Section \ref{Section:Nilpotent} we construct an exponent semigroup representation for every subsemigroup of $\N$.

\medskip\noindent\textbf{Acknowledgments.}
We thank Lenny Fukshansky for simplifying our proof of (c) $\Rightarrow$ (a) in Theorem \ref{Theorem:TFAE}
and Mohamed Omar for a helpful conversation about symmetric functions.  We thank the anonymous referee for
several helpful suggestions, including streamlining the proof of Theorem \ref{Theorem:Nilpotent}.

\section{Preliminaries}\label{Section:Preliminaries}
This section contains some preliminary results and basic examples (see Table \ref{Table:Examples}).
Although elementary, they set the stage for what is to come.

\begin{proposition}\label{Proposition:DimOne}
$\{0\}$ and $\N$ are the only subsemigroups of $\N$ with matricial dimension $1$.
That is, a proper, nontrivial semigroup in $\N$ has matricial dimension at least $2$.
\end{proposition}

\begin{proof}
For a $1 \times 1$ rational matrix $[a]$, we have
$
\S([a]) = 
\begin{cases}
\{0\} & \text{if $a \notin \Z$},\\
\N & \text{if $a \in \Z$}. 
\end{cases}
$
\end{proof}

\begin{proposition}\label{Proposition:DimTwo}
The following semigroups have matricial dimension $2$.
\begin{enumerate}[leftmargin=*]
\item Cyclic semigroups; that is, semigroups of the form $m\N = \semigroup{m}$, in which $m \geq 2$.
\item $\{0\} \cup (m+\N) = \{0,m,m+1,\ldots\}$, in which $m \geq 2$.
\item $\semigroup{2,k}$, in which $k \geq 3$ is odd.
\end{enumerate}
\end{proposition}

\begin{proof}
\noindent(a) If $A = \minimatrix{1}{\frac{1}{m}}{0}{1}$, then $A^n = \minimatrix{1}{\frac{n}{m}}{0}{1} \in \M_2(\Z)$ if and only if $m \mid n$.

\medskip\noindent(b) If $A = \minimatrix{2}{2^{-m}}{0}{0}$, then $A^n = \minimatrix{2^n}{2^{n-m}}{0}{0} \in \M_2(\Z)$ if and only if $n \geq m$.

\medskip\noindent(c) If 
$A=\Big[\begin{smallmatrix}
        0 & 2^{-\lfloor \frac{k}{2} \rfloor}\\
        2^{\lfloor \frac{k}{2} \rfloor+1} &0\\
\end{smallmatrix}\Big]$, then
\begin{equation*}
A^n = 
\begin{cases}
\minimatrix{2^{\frac{n}{2}}}{\0}{\0}{2^{\frac{n}{2}}} & \text{if $n$ is even}, \\[10pt]
\minimatrix{0}{2^{\lfloor \frac{n}{2} \rfloor -\lfloor \frac{k}{2} \rfloor}}{2^{\lfloor \frac{n}{2} \rfloor -\lfloor \frac{k}{2} \rfloor+1}}{0} & \text{if $n$ is odd}.
\end{cases}
\qedhere
\end{equation*}
\end{proof}

We can say more about cyclic semigroups; see Theorem \ref{Theorem:Cyclic}.


\begin{table}\small
\begin{equation*}
\begin{array}{c|ccc}
S  & A & \mdim S & \text{Remarks}\\[2pt]
\hline
\{0\} & [ \tfrac{1}{2} ] & 1 &  \text{Prop.~\ref{Proposition:DimOne}}\\[3pt]
\N & 0 & 1 &  \text{Prop.~\ref{Proposition:DimOne}}\\[3pt]
\semigroup{m} = m\N & \minimatrix{1}{ \frac{1}{m} }{0}{1} & 2 & \text{Prop.~\ref{Proposition:DimTwo}}\\[10pt]
\text{$\semigroup{2,k}$, $k \geq 3$ odd} & \minimatrix{0}{2^{-j}}{2^{j+1}}{0}& 2 & \text{Prop.~\ref{Proposition:DimTwo}}\\[10pt]
 \{0\} \cup \{n,n+1,\ldots\} & \minimatrix{2}{2^{-n}}{0}{0} & 2 &\text{Prop.~\ref{Proposition:DimTwo}}\\[10pt]
\semigroup{3,5,7} & \left[\begin{smallmatrix}
        \frac{-1}{4} & \frac{19}{16}\\ -3 & \frac{-7}{4}
    \end{smallmatrix}\right] & 2 & \text{Rem.~\ref{Remark:357}}\\[12pt] 
\semigroup{3,4} &  \left[\begin{smallmatrix}
        \frac{-9}{4} & \frac{-5}{4} & \frac{1}{8}\\
        -3 & 0 & \frac{-5}{2}\\
        \frac{-1}{2} & \frac{3}{2} & \frac{-7}{4}\\
    \end{smallmatrix} \right] & 3 & \text{Ex.~\ref{Example:34}}\\[15pt]
\semigroup{4,6,17}
&
\left[
\begin{smallmatrix}
 \frac{1563}{16} & -\frac{6933}{16} & -\frac{2295}{16} & \frac{1277}{16} \\
 \frac{393}{8} & -\frac{2759}{8} & -\frac{1085}{8} & \frac{575}{8} \\
 -\frac{1939}{8} & \frac{2975}{16} & -\frac{355}{4} & \frac{97}{4} \\
 -\frac{4605}{16} & -\frac{8063}{8} & -\frac{11517}{16} & \frac{5375}{16} \\
\end{smallmatrix}
\right]
& 4 & \text{Thm.~\ref{Theorem:EmbeddingSharp}}\\[25pt]
\semigroup{5,33,52}& \tiny \left[\begin{smallmatrix}
 -\frac{8589934595}{8192} & \frac{2147467265}{2048} & -\frac{17179738119}{8192} & \frac{16383}{2048} & \frac{65533}{8192} \\[2pt]
 -\frac{5368512511}{1024} & \frac{10737041407}{2048} & -\frac{21474279423}{2048} & \frac{188417}{1024} & \frac{188417}{1024} \\[2pt]
 -\frac{4294770687}{2048} & \frac{8589574143}{4096} & -\frac{17179344895}{4096} & \frac{180225}{2048} & \frac{180225}{2048} \\[2pt]
 -\frac{524289}{4096} & \frac{851969}{8192} & -\frac{1179649}{8192} & -\frac{425985}{4096} & -\frac{425985}{4096} \\[2pt]
 -\frac{8588361721}{8192} & \frac{4294311933}{4096} & -\frac{17178034167}{8192} & \frac{163841}{1024} & \frac{1310727}{8192} \\
    \end{smallmatrix}\right] 
    & 5 & \text{Rem.~\ref{Remark:53352}}\\[30pt]
\semigroup{5,7} & 
\left[
\begin{smallmatrix}
 -\frac{767}{16} & \frac{2937}{16} & -\frac{893}{16} & \frac{383}{8} & \frac{767}{16} \\
 -\frac{161}{4} & \frac{1029}{8} & -\frac{321}{8} & \frac{321}{8} & \frac{129}{4} \\
 -\frac{515}{8} & \frac{3599}{16} & -\frac{1027}{16} & \frac{1027}{16} & \frac{515}{8} \\
 -\frac{255}{8} & \frac{2427}{16} & -\frac{639}{16} & \frac{511}{16} & \frac{383}{8} \\
 \frac{1031}{16} & -\frac{1611}{8} & \frac{515}{8} & -\frac{1029}{16} & -\frac{775}{16} \\
\end{smallmatrix}
\right]
& 5 & \text{Cor.~\ref{Corollary:TwoGenerator}} \\[25pt]
\semigroup{6,9,20} &  \eqref{eq:NuggetMatrix} & 6 & \text{Ex.~\ref{Example:Nugget}}\\
\end{array}
\end{equation*}
\caption{Semigroups $S$ and $A \in \M_d(\Q)$ such that $\S(A) = S$, with references that justify the exact value of the matricial dimension.  Empirically, computer searches among matrices whose entries have denominators that are powers of $2$ yield many nontrivial examples.}
\label{Table:Examples}
\end{table}

In what follows, $\otimes$ denotes the Kronecker product of matrices \cite[Sec.~3.4]{SCLA}.

\begin{theorem}\label{Theorem:Properties}
Let $A$ and $B$ be rational matrices.
\begin{enumerate}[leftmargin=*]
\item $\S(A\oplus B) = \S(A) \cap \S(B)$.
\item $\S(A) \cap \S(B) \subseteq \S(A \otimes  B)$.
\item $\S(\frac{1}{m} A) \subseteq \S(A) \subseteq \S(mA)$ for $m \in \N \setminus \{0\}$.
\item $\S(A) \subseteq \S(A^m)$ for $m \in \N$.
\item If $\ell | \gcd \S(A)$, then $\ell \S(A^{\ell}) = \S(A)$.
\item If $AB=BA$, then $\S(A) \cap \S(B) \subseteq \S(AB)$.
\item $\S(A) \S(B) \subseteq \S(A \otimes B)$.
\item $\S(A^{\mathsf{T}}) = \S(A)$.
\end{enumerate}
\end{theorem}

\begin{proof}
\noindent(a) This holds because $(A \oplus B)^n = A^n \oplus B^n$ for $n \in \N$.

\medskip\noindent(b) This holds because $(A \otimes B)^n = A^n \otimes B^n$ for $n \in \N$.
Note that $(A \otimes B)^n$ may be integral even if $A$ and $B$ are not.  For example, consider
$A = [\frac{2}{3}]$ and $B = [\frac{3}{2}]$.

\medskip\noindent(c) If $A^n \in \M_d(\Z)$, then $(mA)^n = m^n A^n \in \M_d(\Z)$.
Similarly, if $(A/m)^n \in \M_d(\Z)$, then $A^n \in \M_d(\Z)$.

\medskip\noindent(d) If $A^n \in \M_d(\Z)$, then $(A^m)^n = (A^n)^m \in \M_d(\Z)$.

\medskip\noindent(e) If $n \in \ell \S(A^{\ell})$, then $n = \ell i$ for some $i\in \S(A^{\ell})$.
Thus, $A^n = A^{\ell i} = (A^{\ell})^i \in \M_d(\Z)$, so $n \in \S(A)$.
Conversely, if $n \in \S(A)$, then $\ell | n$, so $n = \ell (n/\ell)$ in which
$(A^{\ell})^{n/\ell} = A^n \in \M_d(\Z)$.  Thus, $n \in \ell \S(A^{\ell})$.

\medskip\noindent(f) If $AB = BA$ and $A^n,B^n \in \M_d(\Z)$, then $(AB)^n = A^n B^n \in \M_d(\Z)$.

\medskip\noindent(g) If $s \in \S(A)$ and $t \in \S(B)$, then 
$(A\otimes B)^{st} = (A^s)^t \otimes (B^t)^s$ is integral.  Alternately,
$st \in \S(A) \cap \S(B)$, so this follows from (b).

\medskip\noindent(h) A matrix is integral if and only if its transpose is.
\end{proof}

Let $\SL_d^{\pm}(\Z)$ denote the set of $d \times d$ integer matrices with determinant $\pm 1$.
Then $\SL_d^{\pm}(\Z)$ is closed under inversion since $A (\adj A) = (\adj A) A = (\det A )I$,
in which $\adj A$ denotes the adjugate of a square matrix $A$ \cite[(C.4.3)]{SCLA}.

\begin{proposition}
Let $A,B \in \M_d(\Q)$ and $A = QBQ^{-1}$, in which $Q \in \SL^{\pm}_d(\Z)$.  Then $\S(A) = \S(B)$.
\end{proposition}

\begin{proof}
Let $Q \in \SL^{\pm}_d(\Z)$.  Then $\adj Q \in \M_d(\Z)$ and $Q(\adj Q) = (\adj Q) Q = (\det Q)I = \pm I$, so $Q^{-1} \in \SL_d^{\pm}(\Z)$.
If $n \in \S(B)$, then $A^n = QB^n Q^{-1} \in \M_d(\Z)$, so $n \in \S(A)$.  Thus, $\S(B) \subseteq \S(A)$.
Symmetry provides the reverse containment.
\end{proof}

The converse of the previous proposition is false since $\S(0) = \S(I) = \N$, but 
$0$ and $I$ are not similar (over $ \SL^{\pm}_d(\Z)$ or otherwise).

\section{Power-integral matrices}\label{Section:PowerIntegral}
The next result places severe restrictions on rational matrices that generate nontrivial exponent semigroups.
We first need a few facts from algebraic number theory.
An \emph{algebraic integer} is a complex number that is a root of a monic polynomial with integer coefficients \cite[p.~43]{Stewart}.
For example, $i$,  $\sqrt{2}$, and $47$ are algebraic integers, being zeros of 
$x^2+1$, $x^2-2$, and $x-47$, respectively.
Let $\Z[x]$ (respectively, $\Q[x]$) denote the ring of polynomials with coefficients in $\Z$ (respectively, $\Q$).
The characteristic polynomial of $A \in \M_d(\Q)$ is the monic, degree-$d$ polynomial $p_A(x) = \det (xI - A) \in \Q[x]$.

\begin{theorem}\label{Theorem:CharPoly}
Let $A \in \M_d(\Q)$.  If $\S(A) \neq \{0\}$, then $p_A(x) \in \Z[x]$.
\end{theorem}

\begin{proof}
Suppose $A^n \in \M_d(\Z)$ for some $n \geq 1$.  
Since $p_{A^n}(x) \in \Z[x]$, each eigenvalue of $A^n$ is an algebraic integer.
The eigenvalues of $A^n$ are the $n$th powers of the eigenvalues of $A$,
repeated according to multiplicity \cite[Cor.~11.1.4]{SCLA}.
Since the $n$th root of an algebraic integer
is an algebraic integer \cite[Thm.~2.10]{Stewart}, the eigenvalues of $A$ are algebraic integers.
The coefficients of $p_A(x) \in \Q[x]$ are sums and products of the eigenvalues of $A$, so
they are algebraic integers \cite[Thm.~2.9]{Stewart}, and hence, because they are also rational numbers, integers (in the standard sense) \cite[Lem.~2.14]{Stewart}.   Thus, $p_A(x) \in \Z[x]$.
\end{proof}

\begin{example}
The converse of the previous theorem is false:
$A = \big[\begin{smallmatrix} 2 & \frac{1}{2} \\ 0 & 1 \end{smallmatrix}\big] \in \M_2(\Q)$ satisfies
$p_A(x) = (x-2)(x-1) \in \Z[x]$ and $\S(A) = \{0\}$ since
\begin{equation*}
A^n =  \begin{bmatrix} 2^n & \frac{2^n-1}{2} \\[2pt] 0 & 1 \end{bmatrix} \notin \M_2(\Z)
\quad \text{for any $n \geq 1$}.
\end{equation*}
\end{example}

We say $A \in \M_d(\Q)$ is \emph{power integral} if $A^n \in \M_d(\Z)$ for some $n \geq 1$.
Clearly, $A$ is power integral if and only if $\S(A) \neq \{0\}$.  
Much of the next result is mathematical folklore and portions are well known in various circles.
Recall that the minimal polynomial $m_A(x)\in \Q[x]$ of $A \in \M_d(\Q)$ is the monic polynomial of least
positive degree that annihilates $A$.  Below $\tr(\cdot)$ refers to the trace of the given matrix.

\begin{theorem}\label{Theorem:TFAE}
Let $A \in \M_d(\Q)$.  The following are equivalent.
\begin{enumerate}[leftmargin=*]
\item $p_A(x) \in \Z[x]$.
\item $m_A(x) \in \Z[x]$.
\item Every eigenvalue of $A$ is algebraic integer.
\item There exists an integer $m \geq 1$ such that $mA^n \in \M_d(\Z)$ for all $n \geq 0$.
\item $A = QBQ^{-1}$, in which $Q \in \M_d(\Z)$ is invertible and $B \in \M_d(\Z)$.
\item $\tr (A^k) \in \Z$ for all $k \in \N$.
\end{enumerate}
\end{theorem}

\begin{proof}
\medskip\noindent(a) $\Rightarrow$ (b)
Since $A \in \M_d(\Q)$, the minimal polynomial $m_A(x)$ has coefficients in $\Q$ and divides $p_A(x)$ in $\Q[x]$.  
Since both polynomials are monic and $p_A(x) \in \Z[x]$ by hypothesis, Gauss' lemma ensures that $m_A(x) \in \Z[x]$ \cite[Lem., p.14]{Marcus}.

\medskip\noindent(b) $\Rightarrow$ (c)
Every eigenvalue of $A$ is a zero of $m_A(x)$ \cite[Thm.~11.3.1.d]{SCLA}.  Since $m_A(x) \in \Z[x]$ is monic,
every eigenvalue of $A$ is an algebraic integer.

\medskip\noindent(c) $\Rightarrow$ (a) 
The hypothesis ensures the zeros of $p_A(x) \in \Q[x]$ are algebraic integers.
Sums and products of algebraic integers are algebraic integers \cite[Thm.~2.9]{Stewart}.
The coefficients of $p_A(x)$ are elementary symmetric polynomials of the zeros of $p_A(x)$, so they are algebraic integers.
Since rational algebraic integers are integers in the usual sense \cite[Lem.~2.14]{Stewart}, we conclude $p_A(x) \in \Z[x]$.

\medskip\noindent(a) $\Rightarrow$ (d) 
Suppose $p_A(x) = x^d + c_{d-1} x^{d-1} + \cdots + c_1 x + c_0 \in \Z[x]$.
Let $m \geq 1$ be an integer such that $mA^i \in \M_d(\Z)$ for $i=1,2,\ldots,d-1$.
The Cayley--Hamilton theorem \cite[Thm.~11.2.1]{SCLA}  ensures that $p_A(A) = 0$, so
$A^{d+i} = -c_{d-1} A^{d+i-1} - \cdots - c_1 A^{1+i} - c_0 A^i$
for $i \geq 0$.  Induction confirms that $mA^n \in \M_d(\Z)$ for all $n \geq 0$.

\medskip\noindent(d) $\Rightarrow$ (e) Suppose $mA^n \in \M_d(\Z)$ for all $n\geq 0$. Then 
\begin{equation*}
\Lambda = \{ \vec{x} \in \Z^d : \text{$A^n\vec{x} \subseteq \Z^d$ for all $n \geq 0$}\}
\end{equation*}
is a full-rank sublattice of $\Z^d$ since $m \Z^d \subseteq \Lambda \subseteq \Z^d$.
Let the columns of $Q \in \M_d(\Z)$ comprise a (lattice) basis of $\Lambda$.
Since $A \Lambda \subseteq \Lambda$, there is a $B \in \M_d(\Z)$ such that
$AQ = QB$; that is, $A = QBQ^{-1}$.

\medskip\noindent(e) $\Rightarrow$ (a) Since $A$ and $B$ are similar, $p_A(x) = p_B(x) \in \Z[x]$ \cite[Thm.~10.3.1]{SCLA}.

\medskip\noindent(e) $\Rightarrow$ (f) Since $A$ and $B$ are similar and $B \in \M_d(\Z)$, we have
$\tr(A^k) = \tr(B^k) \in \Z$ for all $k \in \N$ \cite[Thm.~2.6.8]{SCLA}.

\medskip\noindent(f) $\Rightarrow$ (c) 
Suppose toward a contradiction that not every eigenvalue of $A \in \M_d(\Q)$ is an algebraic integer.
Let $\lambda_1,\lambda_2,\ldots,\lambda_d$ denote the eigenvalues of $A$, repeated according to multiplicity and
let $\K = \Q(\lambda_1,\lambda_2,\ldots,\lambda_d)$.  Let $\p$ be a prime ideal of $\K$ with valuation ring $\P$
and let $\lambda_1, \lambda_2, \ldots, \lambda_n$, in which $0 \leq n \leq d$, denote the eigenvalues of $A$ that
do not belong to $\P$; that is, that have negative $\p$-valuation. We may assume $n \geq 1$ since otherwise there is nothing to prove.
Since $\tr(A^k) \in \Z \subseteq \P$ and $\lambda_{n+1},\lambda_{n+2},\ldots, \lambda_d \in \P$, it follows that
\begin{align*}
p_k 
= p_k(\lambda_1,\lambda_2,\ldots,\lambda_n) 
= \lambda_1^k + \lambda_2^k + \cdots + \lambda_n^k
= \tr (A^k) - (\lambda_{n+1}^k + \cdots + \lambda_d^k) ,
\end{align*}
the power-sum symmetric function of degree $k$ in $\lambda_1,\lambda_2,\ldots,\lambda_n$, belongs to $\P$.
The generating function for the elementary symmetric polynomials
\begin{equation*}
e_k = e_k(\lambda_1,\lambda_2,\ldots,\lambda_n) 
= \sum_{1 \leq i_1 < i_2 < \cdots < i_n \leq k} \lambda_{i_1} \lambda_{i_2} \cdots \lambda_{i_k}
\end{equation*}
is $E(t) = \sum_{k=0}^{n} e_k t^k = \prod_{i =1}^n (1 + \lambda_i t)$ \cite[(2.2)]{Macdonald}.
It is related to 
\begin{equation*}
P(t) = \sum_{k=1}^{\infty} p_k t^{k-1} = \sum_{i= 1}^n \frac{\lambda_1}{1- \lambda_i t} 
= \frac{d}{dt} \sum_{i = 1}^n \log \frac{1}{1-\lambda_i t},
\end{equation*}
the generating function
for the power-sum symmetric functions, via the formula $P(-t) = \frac{d}{dt} \log E(t)$ \cite[(2.10')]{Macdonald}.  
Since $E(0) = 1$ and $P(0) = 0$, we have
\begin{align*}
E(t) 
&= \exp\bigg( \int_0^t P(-s)\,ds \bigg)
= \exp\bigg( \sum_{r=1}^{\infty} (-1)^r p_r \frac{t^r}{r} \bigg) \\
&=1-p_1 t+\frac{1}{2} \left(p_1^2+p_2\right) t^2+\frac{1}{6} \left(-p_1^3-3 p_2 p_1-2 p_3\right) t^3+\cdots
\end{align*}
To be more specific,
\begin{equation*}
e_{k}={\frac {(-1)^{k}}{k!}}B_{k}\big(-p_{1},\, -1!p_{2},\, -2!p_{3},\ldots ,\, -(k-1)! p_{k} \big),
\end{equation*}
in which $B_n(x_1,x_2,\ldots,x_k) \in \Z[x_1,x_2,\ldots,x_k]$ is the complete exponential Bell polynomial.
Then for each $k \in \N$,
\begin{align*}
n! (\lambda_1 \lambda_2 \cdots \lambda_n)^k
&= n! (\lambda_1^k \lambda_2^k \cdots \lambda_n^k) 
=n! e_n(\lambda_1^k, \lambda_2^k ,  \ldots, \lambda_n^k) \\
&= (-1)^n  B_n\big(-p_{1},\, -1!p_{2},\, -2!p_{3},\ldots ,\, -(n-1)! p_{n} \big) 
\end{align*}
belongs to $\P$.
As $k\to\infty$, the $\p$-valuation of $n! (\lambda_1 \lambda_2 \cdots \lambda_n)^k$ is unbounded below, which contradicts the fact
that it belongs to $\P$.  Thus, every eigenvalue of $A$ belongs to $\P$.  Since $\p$ was an arbitrary prime ideal of $\K$,
we conclude that every eigenvalue of $A \in \M_d(\Q)$ is an algebraic integer.
\end{proof}

\begin{corollary}
Let $A \in \M_d(\Q)$.
If any of the following occur, then $\S(A) = \{0\}$:
$\det A \notin \Z$;
$\tr (A^k) \notin \Z$ for some $k \geq 1$;
or
$A$ has a rational, non-integer eigenvalue.
\end{corollary}

\begin{proof}
Theorem \ref{Theorem:TFAE} ensures that any of the conditions above imply $p_A(x) \notin \Z[x]$.
The contrapositive of Theorem \ref{Theorem:CharPoly} yields $\S(A) = \{0\}$.
\end{proof}

\section{Cyclic semigroups}\label{Section:Cyclic}

A semigroup $S \subseteq \N$ is \emph{cyclic} if it is singly generated.
Theorem \ref{Theorem:Cyclic} below provides a simple condition for cyclicity.
We first need an argument from \cite{PRM}.  

\begin{lemma}\label{Lemma:Pigeon}
Let $A = SBS^{-1}$, in which $B,S\in \M_d(\Z)$ with $\det B =\pm 1$ and $\det S \neq 0$.  
Then $A^n \in \M_d(\Z)$ for some $1 \leq n \leq (\det S)^{2d^2}$.
\end{lemma}

\begin{proof}
Let $m = \det S$.  Then $mS^{-1} = \adj S$, so
$m A^k = SB^k (\adj S) \in \M_d(\Z)$ for all $k \in \Z$.
There are $(m^2)^{d^2} = m^{2d^2}$ elements in $\M_d(\Z/m^2\Z)$, so the
pigeonhole principle yields integers $1 \leq k < j \leq m^{2d^2}$ such that 
$mA^j \equiv mA^k \pmod{m^2}$.
Thus, there is an $X \in \M_d(\Z)$ such that
$mA^j= mA^k + m^2 X$.
Therefore,
$A^{j-k} = I + (mA^{-k})X \in \M_d(\Z)$,
in which $1 \leq j-k \leq m^{2d^2}$.
\end{proof}

\begin{theorem}\label{Theorem:Cyclic}
Let $A \in \M_d(\Q)$ and $\det A = \pm 1$. Then $\S(A)$ is cyclic.
Moreover, $\S(A)$ is nontrivial if and only if $p_A(x) \in \Z[x]$.
\end{theorem}

\begin{proof}
First, $A^{-n} = \adj(A^n) / \det(A^n) \in \M_d(\Z)$ since $\det A = \pm 1$.  
Thus, $\widetilde{\S}(A) = \{ n \in \Z : A^n \in \M_d(\Z)\}$ is a subgroup of $\Z$,
and hence cyclic, so $\S(A)$ is cyclic.

If $\S(A)$ is nontrivial, Theorem \ref{Theorem:CharPoly} implies $p_A(x) \in \Z[x]$.
Suppose $p_A(x) \in \Z[x]$.  Theorem \ref{Theorem:TFAE} implies that
$A = SBS^{-1}$, in which $S \in \M_d(\Z)$ is invertible and $B \in \M_d(\Z)$.
Since $\det B = \det A = \pm 1$,
Lemma \ref{Lemma:Pigeon} ensures $A^n \in \M_d(\Z)$ for some $n \geq 1$, so $\S(A) \neq \{0\}$.
\end{proof}

\begin{corollary}
If $A \in \M_d(\Q)$ and $p_A(x)$ is a cyclotomic polynomial, then $\S(A)$ is nontrivial and cyclic.
\end{corollary}

\begin{proof}
Since $p_A(x) \in \Z[x]$ and $\det A = p_A(0) = \pm 1$ (being an integer and the product of certain roots of unity)
Theorem \ref{Theorem:Cyclic} ensures that $\S(A)$ is cyclic.
\end{proof}

\section{Symmetric and pseudosymmetric semigroups}\label{Section:Symmetric}
The proof of (c) $\Rightarrow$ (d) in Theorem \ref{Theorem:TFAE} contains a useful observation, which we expand upon here.
This ``Cayley--Hamilton trick'' yields sharp bounds on the matricial dimension of certain numerical semigroups.

\begin{theorem}\label{Theorem:Cayley}
Let $A \in \M_d(\Q)$.  If $\{ n,n+1,\ldots, n+d-1\} \subseteq \S(A)$, then $g(S) \leq n-1$.
That is, if $\S(A)$ contains $d$ consecutive natural numbers, then $\S(A)$ contains all successive natural numbers.
\end{theorem}

\begin{proof}
Suppose $A \in \M_d(\Q)$ and $\{ n,n+1,\ldots, n+d-1\} \subseteq \S(A)$.
Theorem \ref{Theorem:CharPoly} ensures that $p_A(x) \in \Z[x]$.  Write
$p_A(x) = x^d + c_{d-1} x^{d-1} + \cdots + c_1 x + c_0$, in which $c_0,c_1,\ldots,c_{d-1} \in \Z$.
The Cayley--Hamilton theorem \cite[Thm.~11.2.1]{SCLA} provides
\begin{equation*}
A^d = - c_{d-1} A^{d-1} - \cdots - c_1 A - c_0 I.
\end{equation*}
Multiply by $A^{i}$ for $i\geq 0$ and obtain
\begin{equation*}
A^{i+d} = -c_{d-1} A^{i+(d-1)} - \cdots - c_1 A^{i+1} - c_0 A^{i},
\end{equation*}
which shows that $A^{i+d} \in \M_d(\Z)$ whenever $A^i,A^{i+1},\ldots, A^{i+(d-1)} \in \M_d(\Z)$.
Since $\{ n,n+1,\ldots, n+(d-1)\} \subseteq \S(A)$, induction confirms that $\{n,n+1,\ldots\} \subseteq \S(A)$;
that is, $g(S) \leq n-1$.
\end{proof}

The same argument applies to the minimal polynomial:
if $A \in \M_d(\Q)$ and $\S(A)$ contains $\deg m_A(x)$ consecutive
natural numbers, then $\S(A)$ contains all subsequent natural numbers.

\begin{example}\label{Example:34}
Let $S = \semigroup{3,4} = \{3,4,6,7,8,\ldots\}$.
Since $3,4 \in S$ and $5 \notin S$,  Theorem \ref{Theorem:Cayley} says $S \neq \S(A)$ for all $A \in \M_2(\Q)$.  
Thus, $\mdim \semigroup{3,4} = 3$ since $S=\S(A)$ for
\begin{equation*}
A= \begin{bmatrix}
        \frac{-9}{4} & \frac{-5}{4} & \frac{1}{8}\\[2pt]
        -3 & 0 & \frac{-5}{2}\\[2pt]
        \frac{-1}{2} & \frac{3}{2} & \frac{-7}{4}\\
    \end{bmatrix}.
\end{equation*}
\end{example}

A numerical semigroup $S$ is \emph{symmetric} if its Frobenius number $g(S)$ is odd and $x \in \Z \setminus S$ implies $g-x \in S$ \cite[Prop.~4.4(1)]{Rosales}.  
It is \emph{pseudosymmetric} if $g(S)$ is even and $x \in \Z \setminus S$ implies $g-x \in S$ or $x = g(S)/2$ \cite[Prop.~4.4(2)]{Rosales}.
These conditions are usually stated in terms of irreducibility, but for our purposes the definitions above are more convenient.
Recall that the \emph{multiplicity} $m(S)$ of a numerical semigroup $S$ is 
the least generator in the minimal generating set \cite[p.~9]{Rosales}.

\begin{theorem}\label{Theorem:EmbeddingSharp}
$\mdim S \geq m(S)$ for any symmetric numerical semigroup $S$.
\end{theorem}

\begin{proof}
Let $S=\semigroup{n_1,n_2,\ldots,n_k}$ be a minimal presentation of $S$, with $n_1 = m(S)$ and Frobenius number $g=g(S)$.
Then $1,2,3,\ldots,n_1-1 \notin S$, so $g-1,g-2,\ldots,g-(n_1-1) \in S$.  
If $S = \S(A)$ for some $A \in \M_{d}(\Q)$ with $d < n_1$, then Theorem \ref{Theorem:Cayley} implies that 
$g \in S$, which is impossible.  Thus, $\mdim S \geq n_1 = m(S)$.
\end{proof}

\begin{example}\label{Example:Nugget}
Since the McNugget semigroup $\S = \semigroup{6,9,20}$ is symmetric and $m(S) = 6$.
Theorem \ref{Theorem:EmbeddingSharp} and \eqref{eq:NuggetMatrix} ensure that $\mdim S = 6$.
\end{example}

The next corollary states that the matricial dimension of any two-generator numerical semigroup
is bounded below by its minimal generator.

\begin{corollary}\label{Corollary:TwoGenerator}
Let $S$ be a numerical semigroup.
If $e(S) = 2$, then $\mdim S \geq m(S)$.
\end{corollary}

\begin{proof}
This follow from Theorem \ref{Theorem:EmbeddingSharp} since
every numerical semigroup of embedding dimension $2$ is symmetric \cite[Cor.~4.7]{Rosales}.
\end{proof}

\begin{theorem}\label{Theorem:Pseudosymmetric}
Let $S$ be a nontrivial pseudosymmetric numerical semigroup
with Frobenius number $g(S)$ and multiplicity (minimal generator) $m(S)$.
\begin{enumerate}[leftmargin=*]
\item If $g(S) < m(S)$, then $\mdim S = 2$.

\item If $m(S) < g(S) < 2m(S)$, then $\mdim S \geq m(S) - 1$.

\item If $g(S) \geq 2m(S)$, then $\mdim S \geq m(S)$.
\end{enumerate}
\end{theorem}

\begin{proof}
Let $S = \semigroup{n_1,n_2,\ldots,n_k}$ be a minimal presentation of a nontrivial pseudosymmetric
numerical semigroup with $m=m(S) = n_1$ and $g = g(S)$.

\medskip\noindent(a)
If $g < m$, the definition of the Frobenius number implies $g = m-1$ and $S = \{0\} \cup \{m,m+1,\ldots\}$.
Proposition \ref{Proposition:DimTwo} provides an $A \in \M_2(\Q)$ such that $\S(A) = S$.  
We can be more specific: $S = \semigroup{3,4,5}$; see Remark \ref{Remark:345}.

\medskip\noindent(b)
Suppose $m < g < 2m$, so $\frac{g}{2}<m$. 
Then $1,2,3,\ldots,m-1 \notin S$. 
Since $S$ is pseudosymmetric, $g-1,g-2,\ldots,g-(m-1) \in S$ with exactly one exception: $g-i=\frac{g}{2} < m$, which implies that $g-i \notin S$.  Thus, $1,2,\ldots, g-i \notin S$.
If $1 \leq i \leq m-2$, then $g-(m-1) \notin S$, which is impossible since $S$ is pseudosymmetric.
Hence, $i=m-1$ and $\frac{g}{2}=g-(m-1)$. So $g-1,g-2,\ldots,g-(m-2) \in S$ and $g \notin S$. Theorem \ref{Theorem:Cayley} says $\mdim S \geq m-1$.

\medskip\noindent(c)
Suppose $g \geq 2m$.  Then $g-i \neq \frac{g}{2}$ for $1\leq i \leq m-1$.
Since $1,2,3,\ldots,m-1 \notin S$ pseudosymmetry ensures that $g-1,g-2,\ldots,g-(m-1) \in S$. 
Since $g \notin S$, Theorem \ref{Theorem:Cayley} yields $\mdim A \geq m$.
\end{proof}

The next three remarks show that the inequalities in Theorem \ref{Theorem:Pseudosymmetric} are sharp.

\begin{remark}\label{Remark:345}
The only pseudosymmetric numerical semigroup with $g(S)<m(S)$ is $S=\semigroup{3,4,5}$.
Indeed, $S$ is pseudosymmetric if and only if $|\N \setminus S|=(g(S)+2)/2$ \cite[Cor.~4.5]{Rosales}. If $g(S)<m(S)$, then $g(S)=m(S)-1$ and $|\N \setminus S|=m(S)-1$, so
    \begin{equation*}
        |\N \setminus S|=\frac{g(S)+2}{2}
        \iff m(S)-1=\frac{m(S)-1+2}{2}
        \iff m(S)=3.
    \end{equation*}
    Since $m(S)=3$ and $g(S)=3-1=2$, we have $S=\semigroup{3,4,5} = \{0,3,4,\ldots\}$.
\end{remark}

\begin{remark}\label{Remark:357}
The bound in (b) is sharp.  Let $S = \semigroup{3,5,7}$,
which is pseudosymmetric with $g(S)=4$, $m(S) = 3$, and $\mdim S = 2 = m(S)-1$; see Table \ref{Table:Examples}.
\end{remark}

\begin{remark}\label{Remark:53352}
The bound in (c) is sharp.  Let $S=\semigroup{5,33,52}$,
which is pseudosymmetric with $g(S)=94$, $m(S) = 5$, and $\mdim S = 5= m(S)$;
see Table \ref{Table:Examples}.
\end{remark}

\section{Nilpotent representations}\label{Section:Nilpotent}

A numerical semigroup $S \neq \{0\}$ contains every natural number greater than its Frobenius number $g(S)$.
If $A \in \M_{g(S)+1}(\Q)$ is nilpotent, then $n \in \S(A)$ for $n > g(S)$ since $A^n= 0$.  
Thus, we hope to find a nilpotent $A$ such that $\S = S(A)$.

For $\vec{x}=(x_1,x_2,\ldots,x_n) \in \Q^{d-1}$, let $\superdiag \vec{x} \in \M_d(\Q)$ denote the $d \times d$ matrix
with all zero entries except for the first superdiagonal, which has $x_1,x_2,\ldots,x_n$ listed there.
For example, $\superdiag(1,1,\ldots,1)$ is a nilpotent Jordan block.  

\begin{example}
Let $A = \superdiag(a,b,c,d) \in \M_5(\Q)$, in which $a,b,c,d \neq 0$.  Then 
\begin{equation*}
A= \left[
\begin{smallmatrix}
\0 & a & \0 & \0 & \0\\
\0 & \0 & b & \0 & \0\\
\0 & \0 & \0 & c & \0\\
\0 & \0 & \0 & \0 & d\\
\0 & \0 & \0 & \0 & \0\\
\end{smallmatrix}
\right], \,\,
A^2= \left[
\begin{smallmatrix}
\0 & \0 & ab & \0 & \0\\
\0 & \0 & \0 & bc & \0\\
\0 & \0 & \0 & \0 & cd\\
\0 & \0 & \0 & \0 & \0\\
\0 & \0 & \0 & \0 & \0\\
\end{smallmatrix}
\right], \,\,
A^3= \left[
\begin{smallmatrix}
\0 & \0 & \0 & abc & \0\\
\0 & \0 & \0 & \0 & bcd\\
\0 & \0 & \0 & \0 & \0\\
\0 & \0 & \0 & \0 & \0\\
\0 & \0 & \0 & \0 & \0\\
\end{smallmatrix}
\right], \,\,
A^4= \left[
\begin{smallmatrix}
\0 & \0 & \0 & \0 & abcd\\
\0 & \0 & \0 & \0 & \0\\
\0 & \0 & \0 & \0 & \0\\
\0 & \0 & \0 & \0 & \0\\
\0 & \0 & \0 & \0 & \0\\
\end{smallmatrix}
\right],
\end{equation*}
and $A^n = 0$ for $n \geq 5$.  The nonzero entries of the powers of $A$ are certain products of the nonzero entries of $A$.
We leverage this in the proof of the next theorem.  
\end{example}

\begin{theorem}\label{Theorem:Nilpotent}
If $S \subseteq \N$ is a nontrivial numerical semigroup with Frobenius number $g = g(S)$, there is a nilpotent $A \in \M_{g+1}(\Q)$ of order $g+1$ such that $S = \S(A)$. 
\end{theorem}

\begin{proof}
If $S = \N$, let $A$ be the $1 \times 1$ zero matrix.
So suppose $S = \semigroup{n_1,n_2,\ldots,n_k}$, in which $2 \leq n_1 < n_2 < \cdots < n_k$ is a minimal generating set.
Fix $b \in \Z \setminus \{-1,0,1\}$ and define $A = \superdiag(b^{x_1}, b^{x_2},\ldots,b^{x_{g}}) \in \M_{g+1}(\Q)$,
in which $x_1,x_2,\ldots,x_g \in \Z$ must be determined.  Observe that $A$ is nilpotent of order $g+1$ and
\begin{align*}
A^j \in \M_d(\Z)
    &\iff b^{x_i+x_{i+1}+\cdots+x_{i+(j-1)}} \in \Z    &&\text{for $1 \leq i \leq g-(j-1)$} \\
    &\iff x_i+x_{i+1}+\cdots+x_{i+(j-1)} \geq 0            && \text{for $1 \leq i \leq g-(j-1)$}.
\end{align*}
In other words, $A^j$ is an integer matrix if and only if the sum of every $j$ consecutive $x_i$ is nonnegative.
The condition above is additively closed: if it holds for $j$ and $k$, it holds for $j+k$.
Thus, it suffices to find $\vec{x} = (x_1,x_2,\ldots,x_g) \in \Z^g$ such that
\begin{enumerate}[label=(\alph*), leftmargin=*]
    \item $x_1+x_2+\cdots+x_k <0$ for each $k \notin S$; and
    \item $x_i+x_{i+1}+\cdots+x_{i+(n_j-1)} \geq 0$ for each $n_j$ and $1 \leq i \leq g - (n_j-1)$.
\end{enumerate}
Define
\begin{equation*}
x_i =
\begin{cases}
 1   & \text{if $i \in S$ and $i-1 \notin S$},\\
      -1  & \text{if $i \notin S$ and $i-1 \in S$},\\
      0   & \text{otherwise}.
\end{cases}
\end{equation*}
That is, walking along the number line, $x_i = 1$ when entering the semigroup and $x_i = -1$ when leaving the semigroup.
Since $1 \notin S$, we have $x_1 = -1$.

\medskip\noindent\textbf{Property (a).}
The definition of the $x_i$ ensures that (a) holds since
\begin{equation*}
x_1 + x_2 + \cdots + x_k =
\begin{cases}
0 &\text{if $k \in S$} ,\\
-1 & \text{if $k \notin S$}.
\end{cases}
\end{equation*}

\medskip\noindent\textbf{Property (b).}
Observe that
{
\renewcommand{\labelenumi}{(\arabic{enumi})}
\begin{enumerate}[leftmargin=*]
    \item $x_1=-1$ since $S \neq \N$.

    \item $x_i \in \{-1,0,1\}$ for all $1 \leq i \leq g$. 

    \item If $x_i=1$, then $i \in S$.

    \item If $x_i=-1$, then $i \notin S$.

    \item If $x_i=x_j=-1$ and $i < j$, there exists an $i < m < j$ such that $x_m=1$. That is, $\vec{x}$ cannot have two $-1$s with only $0$s between them.

    \item If $x_i=x_j=1$ and $i < j$, there exists an $i < m < j$ such that $x_m=-1$. That is, $\vec{x}$ cannot have two $1$s with only $0$s between them.

    \item If $x_i=0$ and $i \in S$, there exists a $j<i$ such that $x_j=1$ and $x_{j+1}=x_{j+2}=\cdots=x_{i-1}=0$.
     That is, the first nonzero value that precedes $x_i$ in $\vec{x}$ is $1$, with only $0$s between it and $x_i$.

    \item If $x_i=0$ and $i \notin S$, there exists a $j<i$ such that $x_j=-1$ and $x_{j+1}=x_{j+2}=\cdots=x_{i-1}=0$.
     That is, the first nonzero value that precedes $x_i$ in $\vec{x}$ is $-1$, with only $0$s between it and $x_i$.
\end{enumerate}
}

For each fixed generator $n_j$, we use induction to prove that the sum of any $n_j$ consecutive entries of $\vec{x}$ is at least $0$.  This will establish that property (b) holds.

\medskip\noindent\textit{Base Case:} 
Since $n_j \in S$, we have $x_1+x_2+\cdots+x_{n_j}=0$, as desired.

\medskip\noindent\textit{Induction Step:} 
Suppose $x_i+x_{i+1}+\cdots+x_{i+(n_j-1)} \geq 0$ for some $1 \leq i \leq g-n_j$.
We must prove that
\begin{equation}\label{eq:InductionDesire}
x_{i+1}+x_{i+2}+\cdots+x_{i+n_j} \geq 0.  
\end{equation}
The induction hypothesis ensures that
\begin{align*}
x_{i+1}+x_{i+2}+\cdots+x_{i+n_j} 
&= (x_i+x_{i+1}+\cdots+x_{i+(n_j-1)}) - x_i + x_{i+n_j} \\
&\geq  x_{i+n_j} - x_i .
\end{align*}
If $x_i \leq x_{i+n_j}$, then \eqref{eq:InductionDesire} holds.
Since $\vec{x} \in \{-1,0,1\}^g$, this occurs for
\begin{equation*}
(x_i,x_{i+n_j}) \in \big\{(-1,-1), \, (-1,0), \, (-1,1), \, (1,1), \, (0,0), \, (0,1) \big\}.
\end{equation*}
We must consider the three remaining cases:
\begin{equation*}
(x_i,x_{i+n_j}) \in \big\{(1,0), \, (1,-1), \, (0,-1) \big\}.
\end{equation*}

\medskip\noindent\textit{Case 1:}  $(x_i,x_{i+n_j})=(1,-1)$.
Observation 3 says $i \in S$ since $x_i=1$. Thus, $i+n_j \in S$ since $n_j \in S$. 
But  $x_{i+n_j}=-1$ implies $i+n_j \notin S$ by Observation 4, a contradiction. So this case cannot occur.

\medskip\noindent \textit{Case 2:}  $(x_i,x_{i+n_j})=(1,0)$.
Observation 3 says $i \in S$ since $x_i=1$.  Thus, $i+n_j \in S$ since $n_j \in S$. 
Then $1$ is the first nonzero entry of $\vec{x}$ before $x_{i+n_j} = 0$ by observation 7. 
Since $x_i = 1$ and observations 5 and 6 say the nonzero values of $\vec{x}$ alternate between $-1$ and $1$, it follows that $x_{i+1}+x_{i+2}+\cdots+x_{i+n_j} = 0$, so \eqref{eq:InductionDesire} is satisfied.

\medskip\noindent \textit{Case 3:}  $(x_i,x_{i+n_j})=(0,-1)$.
If $i \in S$, then $i+n_j \in S$ since $n_j \in S$, which contradicts $x_{i+n_j}=-1$ by observation 4.
Then $i \notin S$, so observation 8 says $-1$ is the first nonzero entry of $\vec{x}$ before $x_i=0$.
Since $x_{i+n_j} = -1$ and observations 5 and 6 ensure the nonzero values of $\vec{x}$ alternate between $-1$ and $1$, it follows that $x_{i+1}+x_{i+2}+\cdots+x_{i+n_j} = 0$, so  
\eqref{eq:InductionDesire} is satisfied.
\medskip

Thus, $\vec{x}= (x_1,x_2,\ldots,x_g) \in \{-1,0,1\}^g$ satisfies properties (a) and (b), so $A = \superdiag(b^{x_1}, b^{x_2}, \ldots, b^{x_g}) \in \M_{g+1}(\Q)$ satisfies $\S(A) = S$.
\end{proof}

\begin{example}
Let $S = \semigroup{6,9,20}$, which has $g(S) = 43$.  Then
\begin{align*}
\vec{x} &=(-1, 0, 0, 0, 0, 1, -1, 0, 1, -1, 0, 1, -1, 0, 1, -1, 0, 1, -1, 1, 0, -1, 0,\\&\qquad 1, -1, 1, 0, -1, 1, 0, -1, 1, 0, -1, 1, 0, -1, 1, 0, 0, 0, 0, -1).
\end{align*}
\end{example}

\begin{remark}
The order of nilpotence of the matrix in Theorem \ref{Theorem:Nilpotent} is $g(S) + 1$.  Any smaller and it would contradict the
definition of the Frobenius number.
\end{remark}

\begin{corollary}\label{Corollary:General}
If $S \subseteq \N$ is a nontrivial (not necessarily numerical) semigroup, there is a rational matrix $A$ such that $\S(A) = S$.
\end{corollary}

\begin{proof}
Let  $S=\semigroup{n_1,n_2,\ldots,n_k}$, in which $1 \leq n_1 < n_2< \ldots < n_k$ is the unique minimal system of generators of $S$ \cite[Cor.~2.8]{Rosales}
and let $d = \gcd(n_1,n_2,\ldots,n_k)$.  If $d= 1$, then $S$ is a numerical semigroup and Theorem~\ref{Theorem:Nilpotent} provides the desired matrix.
If $d \geq 2$, then $S'=\semigroup{\frac{n_1}{d},\frac{n_2}{d},\ldots,\frac{n_k}{d}}$ is a numerical semigroup \cite[Prop.~2.2]{Rosales}.
The proof of Proposition \ref{Proposition:DimTwo}.a shows that $\S(B) = \semigroup{d}$ for 
$B= \big[ \begin{smallmatrix} 1 & 1/d \\ 0 & 1 \end{smallmatrix} \big]$.
Let $A'$ be the matrix obtained by applying the algorithm of Theorem \ref{Theorem:Nilpotent} to the input $(n_1,n_2,\ldots,n_k)$
and terminating at step $i=d g(S')$.  Although $S$ is not a numerical semigroup, the algorithm still functions and sets $x_i = 0$ if $d \nmid i$.
Then $A = A' \oplus B$ satisfies $\S(A) = S$ by Proposition \ref{Theorem:Properties}.a.
\end{proof}

\begin{example}
Let $S = \semigroup{15,21,33}$, so  $d= \gcd(15,21,33)=3$ and
$B= \big[ \begin{smallmatrix} 1 & 1/3 \\ 0 & 1 \end{smallmatrix} \big]$.
The Frobenius number of $S'=\semigroup{5,7,11}$ is $g(S')=13$, so we run the algorithm until $i = 3 \cdot 13 = 39$ and obtain
\begin{align*}
\vec{x}
&=(-1, 0, 0, 0, 0, 0, 0, 0, 0, 0, 0, 0, 0, 0, 1, -1, 0, 0, 0, 0, \\
&\qquad 1, -1, 0, 0, 0, 0, 0, 0, 0, 1, -1, 0, 1, -1, 0, 1, -1, 0, 0). 
\end{align*}
Then $A=A' \oplus B$ satisfies $\S(A)=S=\semigroup{15,21,33}$. 
\end{example}

\begin{remark}
If $S = \semigroup{n_1,n_2,\ldots,n_k}$ is a minimally presented numerical semigroup with
$n_1 < n_2 < \cdots < n_k$, then there exists an $n_1 \times n_1$ rational matrix $A$ such that $\S(A) = S$.  
That is, $\mdim S \leq m(S)$.
This is supported by numerical evidence, such as the $6 \times 6$ presentation \eqref{eq:NuggetMatrix} of
$\semigroup{6,9,20}$ in the introduction,
and it has been addressed in the sequel preprint \cite{chhabra2024numericalsemigroupsrationalmatrices}.
\end{remark}

\bibliography{NSRM1}
\bibliographystyle{amsplain}

\end{document}